\documentclass[oneside]{amsart}
\usepackage[utf8]{inputenc}
\usepackage{preamble}

\title{A short proof that Rezk's nerve is fully faithful}

\author{Fabian Hebestreit}
\address{Fakultät für Mathematik, Universität Bielefeld,  
Germany}
\email{hebestreit@math.uni-bielefeld.de}

\author{Jan Steinebrunner}
\address{Gonville \& Caius College,
University of Cambridge,
UK}
\email{js2675@cam.ac.uk}

\date{\today}

\begin{document}

\begin{abstract}
    We give a simple proof that complete Segal animae are equivalent to categories.
\end{abstract}
\maketitle
\setcounter{tocdepth}{1}
\tableofcontents

\section{Introduction}

Let $[-] \colon \simp \rightarrow \Cat$ denote the inclusion of non-empty finite ordered sets into categories.%
\footnote{
    We let $\Cat$ denote the $(\infty,1)$-category of $(\infty,1)$-categories and refer to its objects as ``categories''.
    Moreover, we let $\An \subset \Cat$ denote the full subcategory of $(\infty)$-groupoids/spaces and refer to its objects as ``animae''.
}
There is an associated adjunction
\[\ac \colon \sAn \adj \Cat \cocolon \nerve\] 
with the left adjoint $\ac$, short for associated category, given by left Kan extending $[-]$ along the Yoneda embedding $\simp \rightarrow \Fun(\Dop, \An) = \sAn$.
The right adjoint we shall call the Rezk nerve. It is given by 
\[\nerve_n(C) = \Hom_\Cat([n],C) = \core( \Fun([n],C)).\]
The following theorem is one of the foundational results on higher categories, originally (implicitly) suggested by Rezk in \cite{Rezk}.

\begin{thm*}
    The Rezk nerve
    \[
        \nerve \colon \Cat \too \sAn
    \]
    is fully faithful and its essential image is spanned by the complete Segal animae.
\end{thm*}

We recall that a Segal anima is a simplicial anima $X$ for which the Segal maps
\[e_i \colon [1] \longrightarrow [n], \quad 0 \mapsto i-1,\ 1 \mapsto i\]
induce equivalences
\[(e_1, \dots, e_n) \colon X_n \longrightarrow X_1 \times_{X_0} X_1 \times_{X_0} \dots \times_{X_0} X_1.\]
If this is the case one calls a point $f \in X_1$ an equivalence, if there are $\sigma, \tau \in X_2$ with 
\[d_0(\sigma) \simeq f\quad \text{and} \quad d_2(\tau) \simeq f\] and such that both $d_1(\sigma)$ and $d_1(\tau)$ lie in the essential image of the degeneracy $s\colon X_0 \rightarrow X_1$. Denote the collection of components of $X_1$ containing equivalences by $X_1^\sim$ and note that $s$ factors through the inclusion $X_1^\sim \subseteq X_1$ on account of the totally degenerate $2$-simplices in $X$. Then a Segal anima $X$ is called complete if $s \colon X_0 \rightarrow X_1^\sim$ is an equivalence.

In a point-set formulation the theorem was originally proven by Joyal and Tierney in \cite[Section 4]{joyaltierney} and in the formulation above it was established by Lurie in \cite[Corollary 4.3.16]{luriegoo}, but the proof there still makes heavy use of point-set models for both sides. The purpose of the present paper is to give what one might call an invariant proof, that avoids recourse to any particular model for $\Cat$, and that is furthermore rather simple, granted a few key ingredients: Our arguments largely only use fundamental properties of (un)straightening, localisations and Kan extensions. In particular, it is also very different in spirit from the proof of Barwick and Schommer-Pries in \cite{BSP} which characterises both sides by a universal property.

\subsection*{Acknowledgements}
We wish to thank Shaul Barkan for useful discussions, the Department of Mathematical Sciences at the University of Copenhagen for its hospitality during two visits of FH in which this article took shape.

FH was further supported by the German Research Foundation through the research centre `Integral structures in Geometry and Representation theory' (grant no.\ 491392403 - TRR 358) at the University of Bielefeld and JS by the Independent Research Fund Denmark (grant no.\ 10.46540/3103-00099B) and the Danish National Research Foundation through the `Copenhagen Centre for Geometry and Topology' (grant no.\ CPH-GEOTOP-DNRF151).

\section{Recollections on Segal animae}

For any simplicial anima $X$ and any $x, y \in X_0$ we define $\Hom_X(x, y)$ via the pullback:
    \[\begin{tikzcd}
        \Hom_{X}(x,y) \ar[r] \ar[d] & X_{1} \ar[d,"{(}d_{1}{,}d_0{)}"] \\
        \ast \ar[r,"{(}x{,}y{)}"] & X_0 \times X_0.
    \end{tikzcd}\]
In general, there is no composition operation for these $\Hom_X(-, -)$, but there is a canonical span:
\[
    \Hom_X(y, z) \times \Hom_X(x, y) 
    \xleftarrow{(d_0,d_2)} \{(x,y,z)\} \times_{X_0^3} X_2 
    \xrightarrow{d_1} \{(x, z)\} \times_{X_0^2} X_1 = \Hom_X(x, z)
\]
If $X$ is a Segal anima, then the left-pointing map is an equivalence and we do obtain the standard `composition' operation.
Given $f\in \Hom_X(x,y)$ and $g \in \Hom_X(y,z)$ we let $g \circ f \in \Hom_X(x,z)$ denote their composite.
One checks that this is associative (up to homotopy) with units (up to homotopy) given by degenerate edges.
Essentially by definition a point $f \in X_1$ defines an object of $\Hom_X(d_1(f),d_0(f))$ that is invertible for this operation if and only if $f \in X_1^\sim$. 

\begin{lem}\label{nerve-is-CS}
    For any category $C$ its Rezk nerve $\nerve(C)$ is a complete Segal anima and  $\Hom_C(x,y) \simeq \Hom_{\nerve C}(x,y)$ for all $x,y \in C$.
\end{lem}

\begin{proof}
    By iteration, the Segal condition is indeed an immediate consequence of the fact that the diagram
    \[\begin{tikzcd}
        {[}0{]} \ar[r,"0"]\ar[d,"1"] & {[}n{]} \ar[d,"d_{0}"] \\
        {[}1{]} \ar[r,"e_1"] & {[}n+1{]}
    \end{tikzcd}\]
    in $\Cat$ is cocartesian for every $n$. 
    Completeness similarly follows from $[1] \rightarrow [0]$ being a localisation and the observation that $f \in \nerve_1(C)$ is in $\nerve_1(C)^\sim$ if and only if $f$ is an equivalence in $C$ in the usual sense. The final claim follows immediately from the definitions.
\end{proof}

Since fully faithful and essentially surjective functors among categories are equivalences we also learn:

\begin{obs}\label{obs:nerve-conservative}
 The nerve $\nerve \colon \Cat \rightarrow \sAn$ is conservative. 
\end{obs}

Given \cref{nerve-is-CS} and \cref{obs:nerve-conservative}, the triangle identities reduce the proof of the main theorem to showing that the adjunction unit $X \rightarrow \nerve(\ac(X))$ is an equivalence for every complete Segal anima $X$. For a general simplicial anima $X$, it provides us with maps
\[\Hom_X(x,y) \longrightarrow \Hom_{\nerve\ac(X)}(x,y) \simeq \Hom_{\ac(X)}(x,y).\]
The heavy lifting of the proof of the equivalence between categories and complete Segal animae will be in proving that these maps are equivalences for all Segal animae $X$. 
Namely, we will use this in conjunction with:

\begin{lem}\label{equiv-between-CS}
    A map $\varphi\colon X \to Y$ of complete Segal animae is an equivalence if and only if $\pi_0(X_0) \to \pi_0(Y_0)$ is surjective and the induced map
    \[\Hom_X(x,x') \longrightarrow \Hom_Y(\varphi(x),\varphi(x'))\]
    is an equivalence for all $x,x' \in X_0$. 
\end{lem}
\begin{proof}
    We begin by showing that $\varphi$ is conservative, meaning that for any $f \in X_1$ with $\varphi(f) \in Y_1^\sim$ we must already have had $f \in X_1^\sim$.
    Let $x = d_1(f)$ and $y = d_0(f)$ so that $f \in \Hom_X(x, y)$.
    For $\varphi(f)$ to be an equivalence we need a left and right inverse.
    We will focus on the left inverse $g \in \Hom_Y(\varphi(y), \varphi(x))$ which satisfies $g \circ \varphi(f) \simeq \id_{\varphi(x)}$.
    Because $\varphi$ is induces $\pi_0$-surjections on morphism animae we can find $\widetilde{g} \in \Hom_X(y, x)$ mapping to $g$ and this will satisfy that 
    $\varphi(\widetilde{g} \circ f) \simeq g \circ \varphi(f) \simeq \id_{\varphi(x)} \simeq \varphi(\id_x)$.
    Since $\varphi$ is a $\pi_0$-injection on morphism animae we can conclude that $\widetilde{g} \circ f \simeq \id_x$, showing that $f$ has a left inverse.
    One similarly concludes that $f$ has a right inverse and hence is an equivalence.
    This shows that $\varphi$ is conservative, meaning that the middle square in the following diagram is cartesian:
    \[\begin{tikzcd}
        X_0 \ar[r, "s_0"] \ar[d] & X_1^\sim \ar[r, hook] \ar[d] & X_1 \ar[r, "{(d_1, d_0)}"] \ar[d] & X_0 \times X_0 \ar[d] \\
        Y_0 \ar[r, "s_0"] & Y_1^\sim \ar[r, hook] & Y_1 \ar[r, "{(d_1, d_0)}"] & Y_0 \times Y_0. 
    \end{tikzcd}\]
    The right square is cartesian because the assumption precisely means that all induced maps on horizontal fibres are equivalences. 
    In the left square the horizontal maps are equivalences because we assumed that $X$ and $Y$ are complete.
    By pasting we conclude that the entire rectangle is cartesian.
    Comparing the outer horizontal fibres we see that $\varphi\colon X_0 \to Y_0$ induces equivalences of path animae between any two points.
    On the other hand we assumed that $\pi_0(X_0) \to \pi_0(Y_0)$ is surjective, so $X_0 \to Y_0$ must be an equivalence.
    Given this, the right pullback square implies that $X_1 \to Y_1$ is an equivalence and then the Segal condition implies that $X_n \to Y_n$ is an equivalence for all $n$.
\end{proof}

\begin{rem}\label{ac of sSet}
To compare our later description of $\ac(X)$ to previous results in the literature (see \cref{ac literature}), let us also remind the reader of the fact, that
the following two ways of obtaining a category from a simplicial set agree:
\[\begin{tikzcd}[row sep=0]
	& \sAn \\
	{\mathrm{sSet}} && {\Cat.} \\
	& {\mathrm{sSet}[\mathrm{ce}^{-1}] \simeq \mathrm{qCat}[\mathrm{ce}^{-1}]}
	\arrow["\ac", from=1-2, to=2-3]
	\arrow[hook, from=2-1, to=1-2]
	\arrow[from=2-1, to=3-2]
	\arrow[from=3-2, to=2-3]
\end{tikzcd}\]
Here $\mathrm{qCat} \subset \mathrm{sSet}$ denotes the full subcategory spanned by the quasi-categories and in the bottom term we have inverted the categorical equivalences.

This can be seen as follows. One first observes that also the top functor inverts all categorical equivalences: It inverts inner anodynes by direct computation on inner horns and also the inclusion $\{0\} \rightarrow J$, where $J$ is the walking isomorphism (for example by a skeletal induction or as an application of Corollary \ref{cor2} below). The former maps allow one to reduce to equivalences between quasi-categories and this case in turn follows from the above by factoring an equivalence into a trivial cofibration and a trivial fibration (a section of which is again a trivial cofibration) in Joyal's model structure on $\mathrm{sSet}$. 

To produce a natural transformation between the two functors it thus suffices  to compare them on $\mathrm{qCat} \subset \mathrm{sSet}$ on account of the equivalence $\mathrm{qCat}[\mathrm{ce}^{-1}] \simeq \mathrm{sSet}[\mathrm{ce}^{-1}]$. For $X \in \mathrm{qCat}$ we can then consider the map
\[X \cong \Hom_{\mathrm{qCat}}(\Delta^-,X) \Longrightarrow \Hom_{\Cat}([-],X) \simeq \nerve(X)\]
 of simplicial animae, which adjoins to a map $\ac(X) \rightarrow X$, constituting a natural transformation from the composite 
\[\mathrm{qCat} \subset \mathrm{sSet} \subset \sAn \xrightarrow{\ac} \Cat\]
to the localisation map $\mathrm{qCat} \rightarrow \Cat$. This transformation is an equivalence on simplices by direct calculation and the claim then follows by a skeletal induction. 

\end{rem}

\section{The computation of associated categories}

\subsection{The associated category as a localisation}

To approach morphism animae in categories associated to Segal animae, we start by noting that the pointwise formula for left Kan extension gives us
\[\ac(X) = \colim_{[n] \in \simp/X}[n],\]
where the colimit is formed along the composition
\[\simp/X \longrightarrow \simp \xrightarrow{[-]} \Cat.\]
The slice category $\simp/X$ fits into a diagram of categories
\[\begin{tikzcd} 
\simp/X \ar[r] \ar[d] \ar[dr, phantom, very near start, "\lrcorner"] & \sAn/X \ar[d] \ar[r]  \ar[dr, phantom, very near start, "\lrcorner"] & \mathrm{Ar}( \sAn ) \ar[d,"{(}s{,}t{)}"] \\
\simp \times \ast \ar[r,"\mathcal Y"] & \sAn \ar[r, "{(\id,X)}"] & \sAn \times \sAn
\end{tikzcd}\]
where both squares (and their composite) are cartesian, and thus $\simp/X \to \simp$ is a right fibration.

Now recall that the colimit of any functor $F \colon C \rightarrow \Cat$ is given by $\Un(F)[\mathrm{cc}^{-1}]$, see e.g.\ \cite[Corollary 02V0]{Kerodon}, where $\mathrm{cc}$ denotes the collection of cocartesian arrows in the (cocartesian) unstraightening $\Un(F) \rightarrow C$, and that generally there is a canonical cartesian square 
\[\begin{tikzcd}
    \Un(F \circ G) \ar[r] \ar[d] \ar[dr, phantom, very near start, "\lrcorner"] & \Un(F) \ar[d] \\
    D \ar[r] & C
\end{tikzcd}\]
for any $G \colon D \rightarrow C$, both vertical maps being pullbacks of the universal cocartesian fibration $\ast\sslash\Cat \rightarrow \Cat$.

Applied to the situation at hand, we note that by unwinding definitions one finds that $\Un([-])$ is the lax slice category $\ast \sslash \simp$, whose objects are pairs $([n],k)$ with $0 \leq k \leq n$ and morphisms $([n],k) \rightarrow ([m],l)$ are monotone maps $a \colon [n] \rightarrow [m]$ with $a(k) \leq l$. 
Such a morphism $a$ is cocartesian for the forgetful functor $\ast \sslash \simp \rightarrow \simp$ if and only if $a(k) = l$.
The unstraightening of $[-]\colon \simp/X \to \Cat$ is then given by the category $* \sslash \simp /X$ that fits in the cartesian square:
\[\begin{tikzcd}
    * \sslash \simp /X \ar[r] \ar[d] & * \sslash \simp \ar[d] \\
    \simp/X \ar[r] & \simp.
\end{tikzcd}\]
Here, the vertical functors are cocartesian fibrations and the horizontal functors are right fibrations.

In summary we have:

\begin{obs}\label{observationac=lax}
For every simplicial anima $X$, the category $\ac(X)$ is given as the localisation of $* \sslash \simp /X$ at those maps whose image $a \colon ([n],k) \rightarrow ([m],l)$ in $\ast \sslash \simp$ has $a(k) = l$.
\end{obs}

Before we analyse this formula further let us describe the unit of the adjunction $\ac \dashv \nerve$ in terms of it.
Recall that by general nonsense the unit is given by 
\[X \simeq \Hom_{\sAn}(\Delta^-,X) \xrightarrow{\ac} \Hom_{\Cat}([-],\ac(X)) \simeq \nerve(\ac(X))\]
where the left-hand equivalence is induced by the Yoneda embedding
and on the right we use $\ac(\Delta^-) = [-]$.

Now using the pointwise formula for left Kan extensions this translates, naturally in both $[n] \in \simp^{\mathrm{op}}$ and $X \in \sAn$, to the map that takes a morphism $\sigma \colon \Delta^n \rightarrow X$ to the map
\[[n] = \colim_{[k] \in \simp/\Delta^n}[k] \xrightarrow{\sigma} \colim_{[k] \in \simp/X}[k],\]
induced by letting $\sigma$ act on the index categories. Unwinding the unstraightening formula for colimits in $\Cat$, the unit is thus induced by the maps
\[X \simeq \Hom_{\sAn}(\Delta^-,X) \longrightarrow \Hom_\mathrm{Pair}((\ast\sslash\simp/\Delta^-,\mathrm{cc}), (\ast\sslash\simp/X,\mathrm{cc})) \longrightarrow \Hom_\Cat([-], \ac(X))\]
the first coming from the functoriality of slice categories and the second from postcomposition with the localisation map $\ast\sslash\simp/X \longrightarrow \ac(X)$ constructed above then factoring through the localisation 
\[\ast\sslash\simp/\Delta^n \longrightarrow \ac(\Delta^n) \simeq [n],\]
which is explicitly given by 
\[([k],l \in [k], [k] \xrightarrow{\alpha} [n]) \longmapsto \alpha(l).\]
This latter localisation map has a right adjoint:
\[\lambda_n \colon [n] \rightarrow \ast\sslash\simp/\Delta^n, \quad i \longmapsto ([n],i, \mathrm{id}).\]
So for fixed $n$ the map 
\[\Hom_\mathrm{Pair}((\ast\sslash\simp/\Delta^n,\mathrm{cc}), (\ast\sslash\simp/X,\mathrm{cc})) \longrightarrow \Hom_\Cat([n], \ac(X))\]
(and thus the unit) can equally well be described as postcomposition with the localisation of the target and precomposition with $\lambda_n$. 

This description is, however, not quite natural in $n$: For $d \colon [n] \rightarrow [m]$ in $\simp$ the square 
\begin{equation}\label{eq:sq1}
\begin{tikzcd}
    \ast\sslash\simp/\Delta^n \ar[r,"\mathrm{fgt}"] \ar[d, "d_*"'] & {[n]} \ar[d, "d"] \\
    \ast\sslash\simp/\Delta^m \ar[r,"\mathrm{fgt}"] & {[m]}
\end{tikzcd}\end{equation}
is \emph{not} right-adjointable, in other words the Beck-Chevalley transformation $\mathrm{BC}_d$ 
\begin{equation}\label{eq:sq2}
\begin{tikzcd}
	{\ast\sslash\simp/\Delta^n} & {[n]} \\
	{\ast\sslash\simp/\Delta^m} & {[m]}
	\arrow["{\lambda_n}"', from=1-2, to=1-1]
	\arrow["{\lambda_m}", from=2-2, to=2-1]
	\arrow["{d_*}"', from=1-1, to=2-1]
	\arrow["d", from=1-2, to=2-2]
	\arrow["{\mathrm{BC}_d}"{description}, shorten <=3pt, shorten >=3pt, Rightarrow, from=1-1, to=2-2]
\end{tikzcd}\end{equation}
given by the unit and counit
\[\mathrm{BC}_d \colon d_* \circ \lambda_n \Longrightarrow \lambda_m \circ \mathrm{fgt} \circ d_* \circ \lambda_n \simeq \lambda_m \circ d \circ \mathrm{fgt}  \circ \lambda_n \Longrightarrow \lambda_m \circ d\]
is not invertible.

Concretely, in the case at hand it is given by the natural transformation $d_* \circ \lambda_n \Rightarrow \lambda_m \circ d$ whose value on $i \in [n]$ is the morphism
\[
    d_*(\lambda_n(i)) = ([n], i, d\colon \Delta^n \to \Delta^m) \too 
    ([m], d(i), \id_{\Delta^m}) = \lambda_m(d(i))
\]
induced by $d$ itself.
This morphism is cocartesian with respect to $\ast\sslash\simp/\Delta^m \to \simp/\Delta^n$ by the characterisation above. And since being cocartesian for $\ast\sslash\simp/X \rightarrow \simp/X$ generally only depends on the image of a morphism in $\ast\sslash \simp$, we get a well-defined diagram:
\[\begin{tikzcd}
	{\Hom_\sAn(\Delta^m,X)} & {\Hom_{\Cat/(\ast\sslash\simp)}(\ast\sslash\simp/\Delta^m, \ast\sslash\simp/X)} & {\Fun([m],\ast\sslash\simp/X)^{\mathrm{cc}}} & {\Hom_\Cat([m],\ac(X))} \\
	{\Hom_\sAn(\Delta^n,X)} & {\Hom_{\Cat/(\ast\sslash\simp)}(\ast\sslash\simp/\Delta^n, \ast\sslash\simp/X)} & {\Fun([n],\ast\sslash\simp/X)^{\mathrm{cc}}} & {\Hom_\Cat([n],\ac(X))}
	\arrow[from=1-1, to=1-2]
	\arrow[from=1-1, to=2-1]
	\arrow["{\lambda_m^*}", from=1-2, to=1-3]
	\arrow[from=1-2, to=2-2]
	\arrow[from=1-3, to=1-4]
	\arrow[from=1-3, to=2-3]
	\arrow[from=1-4, to=2-4]
	\arrow[from=2-1, to=2-2]
	\arrow["{\mathrm{BC}_d^*}"{description}, shorten <=8pt, shorten >=8pt, Rightarrow, from=2-2, to=1-3]
	\arrow["{\lambda_n^*}"', from=2-2, to=2-3]
	\arrow[from=2-3, to=2-4]
\end{tikzcd}\]
where the superscript $\mathrm{cc}$ denotes the wide subcategory of $\Fun([-],\ast\sslash\simp/X)$ spanned by those transformations which take cocartesian values, so that they are inverted by postcomposition with the localisation $\ast\sslash\simp/X \to \ac(X)$.
The outside rectangle is a priori only a lax square, but as its constituents are animae it really is an ordinary square.

Because pasting squares (\ref{eq:sq1}) and (\ref{eq:sq2}) along $d_*$ yields a trivial square, we find:
\begin{lem}\label{lemma unit}
    The outer rectangle in the above diagram agrees with the naturality square
    \[\begin{tikzcd}
        X_m \ar[r] \ar[d] & \nerve_m(\ac(X)) \ar[d] \\
        X_n \ar[r] & \nerve_n(\ac(X))
    \end{tikzcd}\]
    of the unit, for every $X \in \sAn$ and $d \colon [n] \rightarrow [m]$ in $\simp$.
\end{lem}

\begin{rem}
In fact, the functors $\lambda_n$ assemble into a lax natural transformation $[-] \Rightarrow \ast\sslash\simp/\Delta^-$ of cosimplicial categories, e.g.\ by \cite[Corollary F]{HHLN}, from which one can extract a description of the entire unit, but care would be needed to even make sense of this in a fashion which does not lead to circularity with our main result. Luckily, we will be able to content ourselves with the morphism-wise statement above.
\end{rem}

We now analyse the categories $\ast \sslash \simp/X$ a little further, following an observation of Haugseng \cite[Section 2]{haugseng}. Consider the functor
\[\iota \colon \simp \longrightarrow \ast \sslash \simp, \quad [n] \mapsto ([n],n).\]
It is fully faithful, and one verifies that it has a right adjoint $\kappa$ given by $([n],k) \mapsto [k]$: A map $[m] \rightarrow [k]$ is the same as a map $([m],m) \rightarrow ([n],k)$. The unit of this adjunction is the identity on $\simp$ and the counit is given by the inclusions $([k],k) \rightarrow ([n],k)$, which are cocartesian for $\ast \sslash \simp \rightarrow \simp$ by the analysis above.

\begin{lem}
    This adjunction lifts to an adjunction
    \[\iota_X \colon \simp/X \adj \ast\sslash\simp/X \cocolon \kappa_X\]
    for each simplicial anima $X$.
\end{lem}
\begin{proof}
The functor $\iota_X$ is directly induced by $\iota$ using the fact that the composite $\ast\sslash\simp \rightarrow \simp \rightarrow \sAn$, which is used to define $\ast\sslash\simp/X$, precomposed with $\iota$ is the Yoneda embedding $\simp \rightarrow \sAn$, which is used to define $\simp/X$; put more simply we have 
\[\iota_X(n, x) \simeq ([n],n,x)\]
with $x \in X_n$.
Using that $*\sslash\simp/X$ and $\simp/X$ are right fibrations over $*\sslash\simp$ and $\simp$ one checks that the map
\[
\Hom_{\simp/X}(([n],x),([l],([l] \subseteq [m])^*y)) \too
\Hom_{\ast \sslash \simp/X}(([n],n,x),([m],l,y)) 
\]
given by applying $\iota_X$ and postcomposing with the morphism
$([l], l, ([l] \subseteq [m])^*y) \to ([m], l, y)$
is an equivalence, natural in $(n,x) \in \simp/X$, so that $\iota_X$ admits an adjoint $\kappa_X$ with 
\[\kappa_X([m],l,y) = ([l],([l] \subseteq [m])^*y)\]
as desired, see e.g.\ \cite[Proposition 02FV]{Kerodon}.
\end{proof}

We can now use this adjunction to simplify the formula for associated categories. To this end denote by $\mathrm{lv}$ the collection of morphisms in $\simp$ that preserve the last vertex, i.e.\ the maps $a \colon [n] \rightarrow [m]$ with $a(n) = m$. We shall similarly denote its preimage in $\simp/X$ for any simplicial anima $X$. We then have:

\begin{prop}\label{ac = simp/}
The adjunction $\iota_X \colon \simp/X \adj \ast \sslash\simp/X \cocolon \kappa_X$ descends to an equivalence
\[(\simp/X)[\mathrm{lv}^{-1}] \simeq (* \sslash \simp /X)[\mathrm{cc}^{-1}]\]
for any simplicial anima $X$.
\end{prop}
\begin{proof}
    We first observe that the adjunction descends to an adjunction, since $\iota$ clearly takes last vertex maps to cocartesian ones, and $\kappa$ does the converse. The unit and counit of the resulting adjunctions are induced by the original ones. As the image of an equivalence the unit is thus still an equivalence, and the counit becomes an equivalence in the localisation as it consists of cocartesian edges.
\end{proof}

We obtain the following very useful description of associated categories: 

\begin{cor}\label{cor:localisation}
    For every simplicial anima $X$, we have an equivalence
    \[\ac(X) \simeq (\simp/X)[\mathrm{lv}^{-1}],\]
    natural in the input.
\end{cor}

\begin{rem}\label{ac literature}
\begin{enumerate}
    \item
For a simplicial set $X$ this result is essentially due to Joyal: He showed that the composite of the last vertex map $\mathcal N(\simp/X) \rightarrow X$ of simplicial sets with an inner anodyne to a quasi-category $\widetilde X$ is always a localisation (see e.g.\ \cite[Theorem 1.3]{Stevenson} for a direct proof). 
To translate between these versions of the statement, one uses \cref{ac of sSet}.
\item
    Haugseng proved a version of Corollary \ref{cor:localisation} for complete Segal animae $X$ in \cite[Proposition 2.12]{haugseng}, and our deduction of it from \cref{observationac=lax} is essentially the same as his argument.
\item The full statement of Corollary \ref{cor:localisation} also recently appeared as \cite[Corollary 2.9]{segalification}, but the proof there seems to crucially rely on the equivalence between categories and complete Segal animae.
\end{enumerate}
\end{rem}

Again, we describe the unit in terms of this equivalence. To this end denote by 
\[\mu_n \colon [n] \longrightarrow \simp/\Delta^n, \quad i \mapsto ([i], [i] \subseteq [n])\]
the composite $\kappa_{\Delta^n} \circ \lambda_n$. 

\begin{cor}\label{unitident}
    The naturality square of the unit, which we described in \cref{lemma unit}, also factors as
    \[\begin{tikzcd}
    	{\Hom_\sAn(\Delta^m,X)} & {\Hom_{\Cat/\simp}(\simp/\Delta^m, \simp/X)} & {\Fun([m],\simp/X)^{\mathrm{lv}}} & {\Hom_\Cat([m], \ac(X))} \\
    	{\Hom_\sAn(\Delta^n,X)} & {\Hom_{\Cat/\simp}(\simp/\Delta^n, \simp/X)} & {\Fun([n],\simp/X)^{\mathrm{lv}}} & {\Hom_\Cat([n], \ac(X))}
    	\arrow["{\simp/-}", from=1-1, to=1-2]
    	\arrow[from=1-1, to=2-1]
    	\arrow["{\mu_n^*}", from=1-2, to=1-3]
    	\arrow[from=1-2, to=2-2]
    	\arrow[from=1-3, to=1-4]
    	\arrow[from=1-3, to=2-3]
    	\arrow[from=1-4, to=2-4]
    	\arrow["{\simp/-}", from=2-1, to=2-2]
    	\arrow["{(\kappa_{\Delta^m}\mathrm{BC}_d)^*}"{description}, Rightarrow, from=2-2, to=1-3]
    	\arrow["{\mu_m^*}"', from=2-2, to=2-3]
    	\arrow[from=2-3, to=2-4]
    \end{tikzcd}\]
    Here the superscript $\mathrm{lv}$ denotes the wide subcategory of $\Fun([-],\simp/X)$ spanned by those transformations which consist of last vertex maps, so that they are inverted by postcomposition with $\simp/X \to \ac(X)$.
\end{cor}
\begin{proof}
    Consider the diagram preceding \cref{lemma unit}.
    The two right hand horizontal arrows are given by postcomposing with $\ast\sslash\simp/X \to \ac(X)$, which factors through $\kappa_X\colon \ast\sslash\simp/X \to \simp/X$.
    If we already postcompose with $\kappa_X$ in the two left horizontal arrows, we obtain an factorisation of the unit as:
\[\begin{tikzcd}
	{\Hom_\sAn(\Delta^m,X)} && {\Hom_{\Cat/\simp}(\ast\sslash\simp/\Delta^m, \simp/X)} & {\Fun([m],\simp/X)^{\mathrm{lv}}} & {\Hom_\Cat([m],\ac(X))} \\
	{\Hom_\sAn(\Delta^n,X)} && {\Hom_{\Cat/\simp}(\ast\sslash\simp/\Delta^n, \simp/X)} & {\Fun([n],\simp/X)^{\mathrm{lv}}} & {\Hom_\Cat([n],\ac(X)).}
	\arrow["{(\kappa_X)_* \circ (\ast\sslash\simp/-)}", from=1-1, to=1-3]
	\arrow[from=1-1, to=2-1]
	\arrow["{\lambda_m^*}", from=1-3, to=1-4]
	\arrow[from=1-3, to=2-3]
	\arrow[from=1-4, to=1-5]
	\arrow[from=1-4, to=2-4]
	\arrow[from=1-5, to=2-5]
	\arrow["{(\kappa_X)_* \circ (\ast\sslash\simp/-)}", from=2-1, to=2-3]
	\arrow["{\mathrm{BC}_d^*}"{description}, shorten <=6pt, shorten >=6pt, Rightarrow, from=2-3, to=1-4]
	\arrow["{\lambda_n^*}"', from=2-3, to=2-4]
	\arrow[from=2-4, to=2-5]
\end{tikzcd}\]
    Because the functor $\kappa_Y\colon \ast\sslash\simp/Y \to \simp/Y$ is natural in $Y$,
    there is a commutative square
    \[\begin{tikzcd}
        \Hom_\sAn(\Delta^m, X) \ar[r, "{\simp/-}"] \ar[d, "{\ast\sslash\simp/-}"'] & 
        \Hom_{\Cat/\simp}(\simp/\Delta^m, \simp/X) \ar[d, "{(\kappa_{\Delta^m})^*}"] \\
        \Hom_{\Cat/(\ast\sslash\simp)}(\ast\sslash\simp/\Delta^m, \ast\sslash\simp/X) \ar[r, "{(\kappa_X)_*}"] &
        \Hom_{\Cat/\simp}(\ast\sslash\simp/\Delta^m, \simp/X)
    \end{tikzcd}\]
    and a similar square for $\Delta^n$, which together form a commutative cube.
    Hence the left square in the first diagram of the proof factors as
\[\begin{tikzcd}
	{\Hom_\sAn(\Delta^m,X)} & {\Hom_{\Cat/\simp}(\simp/\Delta^m, \simp/X)} & {\Hom_{\Cat/\simp}(\ast\sslash\simp/\Delta^m, \simp/X)} \\
	{\Hom_\sAn(\Delta^n,X)} & {\Hom_{\Cat/\simp}(\simp/\Delta^n, \simp/X)} & {\Hom_{\Cat/\simp}(\ast\sslash\simp/\Delta^n, \simp/X).}
	\arrow["{\simp/-}", from=1-1, to=1-2]
	\arrow[from=1-1, to=2-1]
	\arrow["{\kappa_{\Delta^m}^*}", from=1-2, to=1-3]
	\arrow[from=1-2, to=2-2]
	\arrow[from=1-3, to=2-3]
	\arrow["{\simp/-}", from=2-1, to=2-2]
	\arrow["{\kappa_{\Delta^n}^*}", from=2-2, to=2-3]
\end{tikzcd}\]
     Inserting this in the diagram, we obtain the claimed factorisation because $\kappa_{\Delta^m} \circ \lambda_m = \mu_m$ and similarly for $n$.
\end{proof}

The transformation $\kappa_{\Delta^m}\mathrm{BC}_d \colon d_* \circ \mu_n \Rightarrow \mu_m \circ d$ of functors $[n] \rightarrow \simp/\Delta^m$ is explicitly given by
\[i \longmapsto \left(d_*\mu_n(i) = ([i],[i] \subseteq [n] \xrightarrow{d} [m]) \xrightarrow{d} ([d(i)], [d(i)] \subseteq [m]) = \mu_m d(i)\right).\]
In particular, this transformation is invertible whenever $d \colon [i] \rightarrow [d(i)]$ is an isomorphism in $\simp$ for all $0 \leq i \leq n$, that is if $d$ is the inclusion of an initial segment, e.g.\ for $d = d_{n+1} \colon [n] \rightarrow [n+1]$; we will use this in Lemma \ref{unit-cartesian} below.

The corollary in particular gives:
\begin{obs}\label{obs:unit-factors}
In simplicial degree $0$ the unit of the adjunction $\ac \dashv \nerve$ is given by factoring
     \[X_0 \simeq [0]/X \subset \simp/X \too (\simp/ X)[\mathrm{lv}^{-1}] \simeq \ac(X)\]
     through $\core(\ac(X)) \simeq \nerve_0(\ac(X))$, naturally in $X \in \sAn$.
\end{obs}

\begin{cor}\label{cor1}
    The map $\pi_0(X_0) \rightarrow \pi_0(\nerve_0(\ac(X))) = \pi_0(\core\  \ac(X))$ induced by the unit of the adjunction $\ac \dashv \nerve$ is surjective for every $X \in \sAn$.
\end{cor}

\begin{proof}
From the previous observation we have a commutative diagram
    \[\begin{tikzcd}
     &   X_0 \ar[rd] \ar[ld] & \\
     \simp/X \ar[rr] && \ac(X)
    \end{tikzcd}\]
   involving the unit as the right arrow. But as a localisation the horizontal map gives a surjection $\pi_0(\core(\simp/X)) \to \pi_0(\core(\ac(X)))$ and for every object $([n], x \in X_n)$ of $\simp/X$ there is the last vertex map $n \colon [0] \rightarrow [n]$, which witnesses that $([n], x \in X_n) \sim ([0], n^*x \in X_0)$ in $\pi_0(\core(\ac(X)))$, which lies in the image of the left map.
\end{proof}

\subsection{Morphism animae in associated categories}
Our eventual goal will be to check that for $X$ a Segal anima the map 
    \[\Hom_X(x, y) \too \Hom_{\nerve(\ac(X))}(x, y) \simeq \Hom_{\ac(X)}(x, y)\]
induced by the unit is an equivalence.

To this end we recall a statement that forms part of the calculus of fractions for localisations; we originally learned this from Cisinski's book \cite{Cisinski}, see also \cite{hyperhyper}.

\begin{lem}\label{hominloc}
Let $C$ be a category and $W$ a wide subcategory. If for some $d \in C$ the functor
\[C^{\mathrm{op}} \longrightarrow \An, \quad c \longmapsto \ \mid c/C \times_C d/W \mid\]
inverts the morphisms of $W$, then the canonical map
\[|c/C \times_C d/W|\ \longrightarrow  p(c)/C[W^{-1}] \times_{C[W^{-1}]} p(d)/\core(C[W^{-1}])  \simeq \Hom_{C[W^{-1}]}(p(c),p(d)),\]
  is an equivalence, where $p \colon C \rightarrow C[W^{-1}]$ is the localisation.
\end{lem}

\begin{proof}
We repeat (and simplify) the proof from \cite{HLS}.

We first claim
\[|c/C \times_C d/W| \simeq \colim_{\substack{g \colon d \rightarrow e \\ \in d/W}} \Hom_{C}(c,e):\]
The colimit is given as the realisation of the unstraightening of the defining functor
\[d/W \xrightarrow{t} C \xrightarrow{\Hom_{C}(c,-)} \An\]
which is indeed $c/C \times_C d/W$.

We will argue by using left Kan extension along the opposite of the localisation map $p \colon C \to C[W^{-1}]$.
First, note that left Kan extension of the representable functor $\Hom_C(-, x)$ along $p^\op$ has to be the representable $\Hom_{C[W^{-1}]}(-, p(x))$ by the Yoneda lemma.

Using this we compute
\begin{align*}
p^{\mathrm{op}}_!(|-/C \times_C d/W|) &\simeq p^{\mathrm{op}}_!\left(\colim_{\substack{g \colon d \rightarrow e \\ \in d/W}} \Hom_{C}(-,e)\right)  \\
& \simeq \colim_{\substack{g \colon d \rightarrow e \\ \in d/W}} p^{\mathrm{op}}_! \Hom_{C}(-,e) 
\simeq \colim_{\substack{g \colon d \rightarrow e \\ \in d/W}} \Hom_{C[W^{-1}]}(-,p(e)) .
\end{align*}
The functor $\Hom_{C[W^{-1}]}(-,p(e))$ over which we take the colimit inverts all morphisms in $d/W$ (by two-out-of-three for equivalences in $C[W^{-1}]$),
so we may equivalently take the colimit over $|d/W|$.
But $d/W$ has $\mathrm{id}_d$ as an initial object, so $|d/W|$ is contractible and we can evaluate the colimit as $\Hom_{C[W^{-1}]}(-, p(d))$.

If now $|-/C \times_C d/W| \colon C^{\mathrm{op}} \rightarrow \An$ inverts $W$, then its left Kan extension along $p^{\mathrm{op}}$ is simply given by descending it 
(the counit $p^\op_!(p^\op)^* \to \id$ is an equivalence because $(p^\op)^*$ is fully faithful),
so we learn
\[|-/C \times_C d/W| \simeq (p^\mathrm{op})^*p^\mathrm{op}_!(|-/C \times_C d/W|) \simeq (p^\mathrm{op})^*\Hom_{C[W^{-1}]}(-,p(d))\]
which we wanted to show.
\end{proof}

We can now apply this to $C= \simp/X$ and $W = (\simp/X)^{\mathrm{lv}}$, the wide subcategory consisting of the last vertex maps. To this end consider for $x \in X_n$ and $y \in X_0$ the pullback
\[\begin{tikzcd}
         P_X(x,y) \ar[r] \ar[d] & X_{n+1} \ar[d,"{(}d_{n+1}{,}{(}n+1{)}^*{)}"] \\
        \ast \ar[r,"{(}x{,}y{)}"] & X_n \times X_0.
    \end{tikzcd}\]
Then for objects $c = ([n], x \in X_n)$ and $d = ([0], y \in X_0)$ of $\simp/X$ we find that the fibre of 
\[\mathrm{fgt}_{c,d} \colon c/(\simp/X) \times_{\simp/X} d/(\simp/X)^{\mathrm{lv}} \longrightarrow [n]/\simp \times_\simp [0]/\simp^{\mathrm{lv}} \simeq [n]/\simp \times_\simp \simp^{\mathrm{lv}}\]
over $d_{n+1}\colon [n]\rightarrow [n+1]$ is exactly $P_X(x,y)$: Commuting pullbacks this is the pullback of the fibres of the three constituent forgetful functors
    \[c/(\simp/X) \rightarrow [n]/\simp, \quad {\simp/X} \rightarrow \simp, \quad \text{and} \quad d/(\simp/X)^{\mathrm{lv}} \too \simp^{\rm lv}
    \]
over $d_{n+1}, [n+1]$ and $[n+1]$, respectively. In the first case we obtain the fibre of $d_{n+1} \colon X_{n+1} \rightarrow X_n$ over $x$, in the second case we get $X_{n+1}$ and in the final one we obtain the fibre of $(n+1)^* \colon X_{n+1} \rightarrow X_0$ over $y$. In total, we find that the fibre of $\mathrm{fgt}_{c,d}$ over $d_{n+1}$ is given by
\[(\{x\} \times_{X_n} X_{n+1}) \times_{X_{n+1}} (X_{n+1} \times_{X_0} \{y\}),\]
which indeed cancels to $P_X(x,y)$.

\begin{prop}\label{criterion verify}
    For a simplicial anima $X \in \sAn$, and objects $c = ([n], x \in X_n)$ and $d = ([0], y \in X_0)$ of $\simp/X$ the map 
    \[P_X(x,y) \longrightarrow |c/(\simp/X) \times_{\simp/X} d/(\simp/X)^{\mathrm{lv}}|\]
    just constructed is an equivalence. 
\end{prop}

Let us defer the proof for a moment and draw the desired corollary. To this end note that for $X$ a Segal anima, we can enlarge the defining diagram by another cartesian square to
\[\begin{tikzcd}
         P_X(x,y) \ar[r] \ar[d] & X_{n+1} \ar[d,"{(}d_{n+1}{,}{(}n+1{)}^*{)}"] \ar[r,"e_{n+1}^*"] & X_1 \ar[d,"{(}d_1{,}d_0{)}"] \\
        \ast \ar[r,"{(}x{,}y{)}"] & X_n \times X_0  \ar[r,"n^* \times \mathrm{id}"] & X_0 \times X_0,
\end{tikzcd}\]
    to find equivalences $P_X(x,y) \simeq \Hom_X(n^*x,y)$, for all $x \in X_n$ and $y \in X_0$. Since this in particular shows that $P_X(x,y)$ only depends on the last vertex of $x$, we can apply \cref{hominloc} to find:

\begin{cor}\label{cor2}
    If $X$ is a Segal anima, then the maps 
    \[
        \Hom_X(x, y) \simeq P_X(x, y) 
        \too |c/(\simp/X) \times_{\simp/X} d/(\simp/X)^{\mathrm{lv}}|
        \xrightarrow{p} \hom_{(\simp/X)[\mathrm{lv}^{-1}]}(c,d) \simeq \hom_{\ac(X)}(x, y)
    \]
    are equivalences for $c=([0],x \in X_0), d=([0],y \in X_0) \in \simp/X$.
\end{cor}

\begin{proof}[Proof of Proposition \ref{criterion verify}]
    Consider the forgetful functor 
    \[
        \mathrm{fgt}_{c,d} \colon c/(\simp/X) \times_{\simp/X} d/(\simp/X)^{\mathrm{lv}}
        \too [n]/\simp \times_{\simp} [0]/\simp^{\rm lv} \simeq [n]/\simp \times_{\simp} \simp^{\rm lv},
    \]
    where the equivalence comes from the fact that $[0]$ is initial in $\simp^{\mathrm{lv}} $. We claim that it is a right fibration.
    
    In general, if $\pi\colon E \to B$ is a right fibration, then so is $e/E \to \pi(e)/B$ and right fibrations are closed under pullback in $\Ar(\Cat)$; both these claims are for example simple consequences of the characterisation of right fibrations as those functors $\pi$ for which $(\Ar(\pi),t) \colon \Ar(E) \rightarrow \Ar(B) \times_B E$ is an equivalence, see e.g.~\cite[Proposition 00TE]{Kerodon}.
    Because $\simp/X \to \simp$ is a right fibration, the first statement implies that all three constituents of the functor above are right fibrations, and the second statement then gives the claim.

    Now the category $[n]/\simp \times_{\simp} \simp^{\rm lv}$ is a wide subcategory of $[n]/\simp$ and one checks that it has $d_{n+1} \colon [n] \rightarrow [n+1]$ as an initial object. 
    But again generally, if $\pi \colon E \rightarrow B$ is a cartesian fibration and $\ast \in B$ is initial, then the inclusion $E_\ast \rightarrow E$ of the fibre over $\ast$ admits a right adjoint, given by sending an object $e \in E$ to the source of a cartesian lift of $\ast \rightarrow \pi(e)$ ending at $e$. In particular, one has $|E| \simeq |E_\ast|$.

    Combining this with the discussion preceding the statement, which identified $\mathrm{fgt}_{c,d}^{-1}(d_{n+1})$ with $P_X(x,y)$, we obtain the claim.
\end{proof}

\section{Proof of the main result}

Our next task is to show that the equivalence
\[\Hom_X(x,y) \too \Hom_{\ac(X)}(x,y) \simeq \Hom_{\nerve(\ac(X))}(x,y)\]
constructed in \cref{cor2} agrees with the map induced by the unit. We will use this in the following form:

\begin{lem}\label{unit-cartesian}
    The square induced by the unit
    \[\begin{tikzcd}
        X_1 \ar[d, "{(d_1, d_0)}"'] \ar[r] &
        \nerve_1(\ac(X)) \ar[d, "{(d_1, d_0)}"] \\
        X_0 \times X_0 \ar[r] &
        \nerve_0(\ac(X)) \times \nerve_0(\ac(X)) 
    \end{tikzcd}\]
    is cartesian for every Segal anima $X$.
\end{lem}

\begin{proof}
  By \cref{unitident} and the discussion thereafter the naturality square 
  \[\begin{tikzcd}
  X_1 \ar[r] \ar[d,"{(d_1,d_0)}"] & \nerve_1(\ac(X)) \ar[d,"{(d_1,d_0)}"] \\
  X_0 \times X_0 \ar[r] & \nerve_0(\ac(X)) \times \nerve_0(\ac(X))
  \end{tikzcd}\]
  can be described in terms of $\simp/X$ as follows: The part involving $d_1$ is directly induced by the commutative square
  \[\begin{tikzcd}
\Hom_\sAn(\Delta^1,X) \ar[r] \ar[d,"{d_1}"] & \Hom_\Cat(\simp/\Delta^1, \simp/X) \ar[r] \ar[d,"{d_1}"] & \Hom_\Cat([1],\simp/X) \ar[d,"{d_1}"] \\
  \Hom_\sAn(\Delta^0,X) \ar[r] & \Hom_\Cat(\simp/\Delta^0, \simp/X) \ar[r] & \Hom_\Cat([0],\simp/X).
  \end{tikzcd}\]
  Unwinding, the top functor is given by 
  \[x \longmapsto \left(([0],d_1(x)) \xrightarrow{d_1} ([1],x)\right).\]
  For the part involving $d_0$, we can reinterpret the lax diagram from \cref{unitident} as a functor 
  \[\Hom_\sAn(\Delta^1,X) \longrightarrow \Hom_{\Cat}(\simp/\Delta^1,\simp/X) \xrightarrow{(\kappa_{\Delta^1}\mathrm{BC}_{d_0})^*} \Hom_\Cat([1],\Fun([0],\simp/X)^{\mathrm{lv}}).\]
It is given by 
\[x \longmapsto \left(([0],d_0(x)) \xrightarrow{d_0} ([1],x)\right)\]
and by definition these two functors glue to a map (where $\mathrm{bw}$ denotes the backward arrow)
\[\Hom_\sAn(\Delta^1,X) \longrightarrow \Hom_\mathrm{Pair}(([1] \cup_{\{1\}} [1],\mathrm{bw}),(\simp/X,\mathrm{lv})),\]
 given by 
\[x \longmapsto \left(([0],d_1(x)) \xrightarrow{d_1} ([1],x) \xleftarrow{d_0} ([0],d_0(x))\right),\]
so that the diagram
  \[\begin{tikzcd}
  \Hom_\sAn(\Delta^1,X) \ar[r] \ar[d,"{(d_1,d_0)}"] & \Hom_\mathrm{Pair}(([1] \cup_{\{1\}} [1],\mathrm{bw}),(\simp/X,\mathrm{lv})) \ar[d,"{(\mathrm{ev}_{0}, \mathrm{ev}_{0'})}"] \\
  \Hom_\sAn(\Delta^0,X) \times \Hom_\sAn(\Delta^0,X) \ar[r] & \Hom_\Cat([0],\simp/X) \times \Hom_\Cat([0],\simp/X)
  \end{tikzcd}\]
  projects to the original square after postcomposing with the localisation $\simp/X \to \ac(X)$.

  But the map induced on fibres is now by construction precisely the one we have shown becomes an equivalence in \cref{cor2}.
\end{proof}

With this proposition established we can finally come to the now short:

\begin{proof}[Proof of the main theorem]
    The adjunction $\ac \dashv \nerve$ restricts to an adjunction between complete Segal animae and categories by \cref{nerve-is-CS} and by \cref{obs:nerve-conservative} the nerve is conservative. By the triangle identities it therefore suffices to show that the unit $X \rightarrow \nerve(\ac(X))$ is an equivalence for every complete Segal anima $X$.
    Indeed, the unit 
    induces equivalences on morphism animae by \cref{unit-cartesian} and is surjective on components of the $0$-simplices by \cref{cor1}. 
    Therefore it is an equivalence by \cref{equiv-between-CS}.
\end{proof}

\bibliographystyle{amsalpha}

\providecommand{\bysame}{\leavevmode\hbox to3em{\hrulefill}\thinspace}

\end{document}